\documentclass[conference]{IEEEtran}
\usepackage{amsmath,tikz,amsfonts,amssymb,amsthm,txfonts,hyperref,graphicx,balance,tikz}
\usepackage{fancyvrb}
\DefineVerbatimEnvironment{code}{Verbatim}{fontsize=\small}
\newtheorem{observation}{Observation}
\newtheorem{lemma}{Lemma}
\newtheorem{theorem}{Theorem}
\newtheorem{conjecture}{Conjecture}
\newtheorem{corollary}{Corollary}
\newcommand{\N}{\mathbb N}
\newcommand{\Z}{\mathbb Z}
\newcommand{\Q}{\mathbb Q}
\newcommand{\K}{\textbf{\textit{K}}}
\newcommand{\LL}{\textbf{\textit{L}}}
\title{Is there a computable upper bound on the heights of rational
solutions of a Diophantine equation with a finite number of solutions?}
\author{
\IEEEauthorblockN{Apoloniusz Tyszka}
\IEEEauthorblockA{
University of Agriculture\\
Faculty of Production and Power Engineering\\
Balicka 116B, 30-149 Krak\'ow, Poland\\
Email: rttyszka@cyf-kr.edu.pl}}
\begin{document}
\date{}
\maketitle
\begin{abstract}
The height of a rational number \mbox{$\frac{p}{q}$} is denoted by \mbox{$h\left(\frac{p}{q}\right)$}
and equals \mbox{$\max(|p|,|q|)$} provided \mbox{$\frac{p}{q}$} is written in lowest terms.
The height of a rational tuple \mbox{$(x_1,\ldots,x_n)$} is denoted by \mbox{$h(x_1,\ldots,x_n)$}
and equals \mbox{$\max(h(x_1),\ldots,h(x_n))$}.
Let \mbox{$G_n=\{x_i+1=x_k: i,k \in \{1,\ldots,n\}\} \cup \{x_i \cdot x_j=x_k: i,j,k \in \{1,\ldots,n\}\}$}.
Let \mbox{$f(1)=1$}, and let \mbox{$f(n+1)=2^{\textstyle 2^{\textstyle f(n)}}$} for every positive integer~$n$.
We conjecture: {\tt (1)} if a system \mbox{${\cal S} \subseteq G_n$} has only finitely
many solutions in rationals \mbox{$x_1,\ldots,x_n$}, then each such solution \mbox{$(x_1,\ldots,x_n)$}
satisfies $h(x_1,\ldots,x_n) \leqslant
\left\{
\begin{array}{lcl}
1 &~& {\rm (if~} n=1{\rm )}\\
2^{\textstyle 2^{n-2}} &~& {\rm (if~} n \geqslant 2{\rm )}
\end{array}
\right.$;
{\tt (2)} if a system \mbox{${\cal S} \subseteq G_n$}
has only finitely many solutions in \mbox{non-negative} rationals \mbox{$x_1,\ldots,x_n$},
then each such solution \mbox{$(x_1,\ldots,x_n)$} satisfies \mbox{$h(x_1,\ldots,x_n) \leqslant f(2n)$}.
We prove: {\tt (1)} both conjectures imply that there exists an algorithm which takes as input a Diophantine equation,
returns an integer, and this integer is greater than the heights of rational solutions, if the solution set is finite;
{\tt (2)} both conjectures imply that the question whether or not a given Diophantine
equation has only finitely many rational solutions is decidable by a single query to an oracle
that decides whether or not a given Diophantine equation has a rational solution.
\end{abstract}
\begin{IEEEkeywords}
Diophantine equation which has only finitely many rational solutions, Hilbert's Tenth Problem for $\Q$,
relative decidability, upper bound on the heights of rational solutions.
\end{IEEEkeywords}
\section{{\bf Introduction}}
\IEEEoverridecommandlockouts
\IEEEPARstart{T}{he} height of a rational number \mbox{$\frac{p}{q}$} is denoted by \mbox{$h\left(\frac{p}{q}\right)$}
and equals \mbox{$\max(|p|,|q|)$} provided \mbox{$\frac{p}{q}$} is written in lowest terms.
The height of a rational tuple \mbox{$(x_1,\ldots,x_n)$} is denoted by \mbox{$h(x_1,\ldots,x_n)$}
and equals \mbox{$\max(h(x_1),\ldots,h(x_n))$}. We attempt to formulate a conjecture which implies
a positive answer to the following open problem:
\vskip 0.2truecm
\noindent
{\em Is there an algorithm which takes as input a Diophantine equation, returns an integer,
and this integer is greater than the heights of rational solutions, if the solution set is finite?}
\section{{\bf Conjecture~\ref{con1} and its equivalent form}}
\begin{observation}\label{obs0}
Only \mbox{$x_1=0$} and \mbox{$x_1=1$} solve the equation \mbox{$x_1 \cdot x_1=x_1$}
in integers (rationals, real numbers, complex numbers). For each integer \mbox{$n \geqslant 2$},
the following system
\begin{displaymath}
\left\{
\begin{array}{rcl}
x_1 \cdot x_1 &=& x_1 \\
x_1+1 &=& x_2 \\
x_1 \cdot x_2 &=& x_2 \\
\forall i \in \{2,\ldots,n-1\} ~x_i \cdot x_i &=& x_{i+1} {\rm ~(if~}n \geqslant 3{\rm)}
\end{array}
\right.
\end{displaymath}
\noindent
has exactly one integer (rational, real, complex) solution, namely
\mbox{$\left(1,2,4,16,256,\ldots,2^{\textstyle 2^{n-3}},2^{\textstyle 2^{n-2}}\right)$}.
\end{observation}
\par
Let
\[
G_n=\{x_i+1=x_k:~i,k \in \{1,\ldots,n\}\} \cup
\]
\[
\{x_i \cdot x_j=x_k: i,j,k \in \{1,\ldots,n\}\}
\]
\begin{conjecture}\label{con1}
If a system \mbox{${\cal S} \subseteq G_n$} has only finitely many solutions in rationals \mbox{$x_1,\ldots,x_n$},
then each such solution \mbox{$(x_1,\ldots,x_n)$} satisfies
\[
h(x_1,\ldots,x_n) \leqslant
\left\{
\begin{array}{lcl}
1 &~& {\rm (if~} n=1{\rm )}\\
2^{\textstyle 2^{n-2}} &~& {\rm (if~} n \geqslant 2{\rm )}
\end{array}
\right.
\]
\end{conjecture}
\par
Observation \ref{obs0} implies that the bound
\[
\left\{
\begin{array}{lcl}
1 && {\rm (if~} n=1{\rm )}\\
2^{\textstyle 2^{n-2}} && {\rm (if~} n \geqslant 2{\rm )}
\end{array}
\right.
\]
cannot be decreased.
\vskip 0.2truecm
\par
Conjecture~\ref{con1} is equivalent to the following conjecture on rational arithmetic:
if rational numbers \mbox{$x_1,\ldots,x_n$} satisfy
\[
h(x_1,\ldots,x_n)>\left\{
\begin{array}{lcl}
1 &~& {\rm (if~} n=1{\rm )}\\
2^{\textstyle 2^{n-2}} &~& {\rm (if~} n \geqslant 2{\rm )}
\end{array}
\right.
\]
then there exist rational numbers \mbox{$y_1,\ldots,y_n$} such that
\[
h(x_1,\ldots,x_n)<h(y_1,\ldots,y_n)
\]
and for every \mbox{$i,j,k \in \{1,\ldots,n\}$}
\[
(x_i+1=x_k \Longrightarrow y_i+1=y_k) ~\wedge (x_i \cdot x_j=x_k \Longrightarrow y_i \cdot y_j=y_k)
\]
\begin{theorem}\label{new2}
Conjecture~\ref{con1} is true if and only if the execution of Flowchart 1 prints infinitely many numbers.
\end{theorem}
\begin{center}
\includegraphics[width=\hsize]{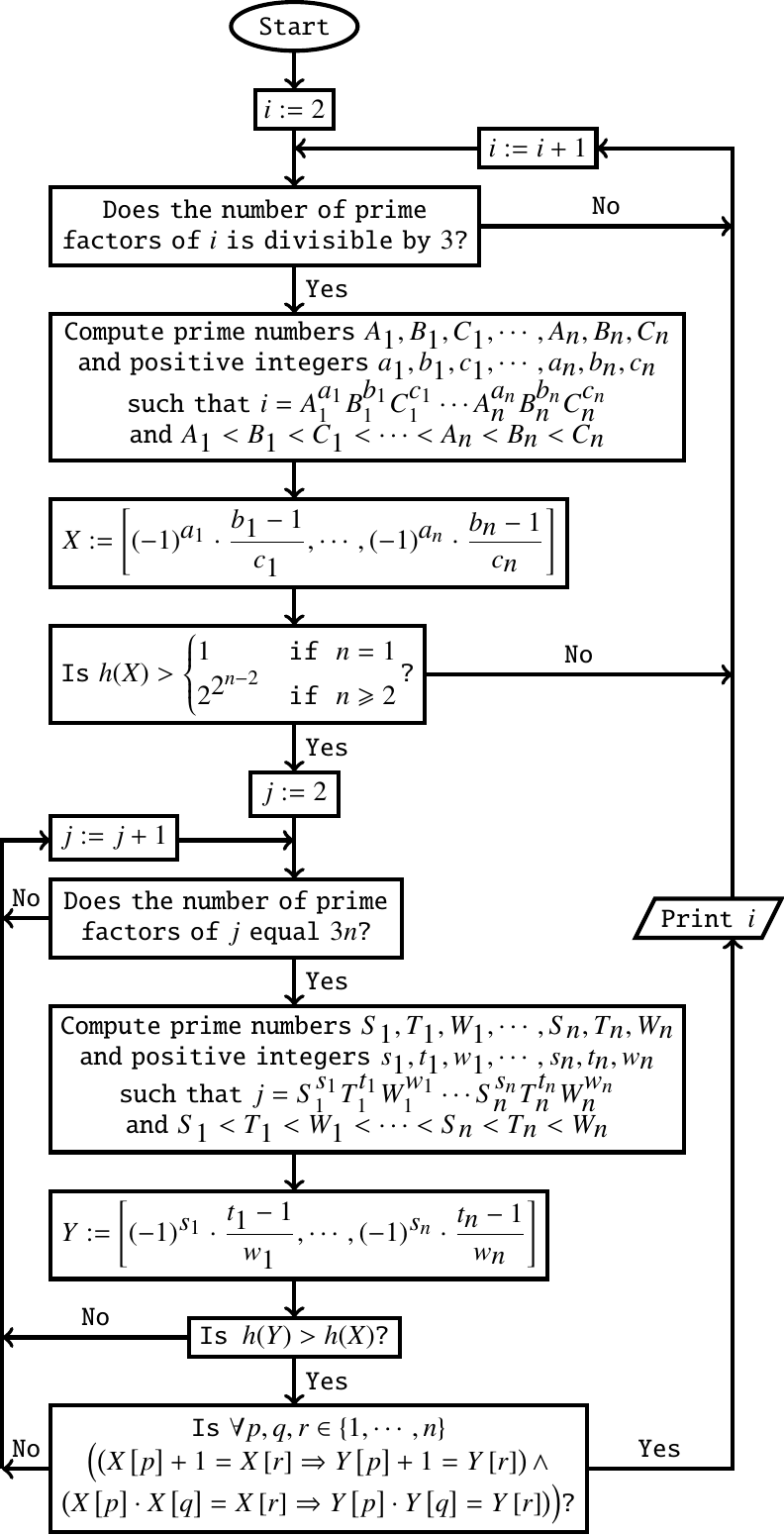}
\end{center}
\vskip 0.01truecm
\centerline{{\it Flowchart 1: An infinite-time computation which}}
\vskip 0.01truecm
\centerline{{\it decides whether or not Conjecture~\ref{con1} is true}}
\begin{proof}
Let $\Gamma_3$ denote the set of all integers \mbox{$i \geqslant 2$} whose number of prime factors is
divisible by $3$. The claimed equivalence is true because the algorithm from Flowchart 1
applies a surjective function \mbox{$\eta \colon \Gamma_3 \to \bigcup_{n=1}^\infty\limits {\Q}^n$}.
\end{proof}
\begin{corollary}
Conjecture~\ref{con1} can be written in the form
\mbox{$\forall x \in \N ~\exists y \in \N ~\phi(x,y)$},
where \mbox{$\phi(x,y)$} is a computable predicate.
\end{corollary}
\section{{\bf Algebraic lemmas -- part 1}}
Let ${\cal R}$ denote the class of all rings, and let \mbox{$\cal{R}${\sl ng}} denote the class of all rings $\K$
that extend $\Z$. Let
\[
E_n=\{1=x_k: k \in \{1,\ldots,n\}\} \cup
\]
\[
\{x_i+x_j=x_k:~i,j,k \in \{1,\ldots,n\}\} \cup
\]
\[
\{x_i \cdot x_j=x_k:~i,j,k \in \{1,\ldots,n\}\}
\]
\begin{lemma}\label{lem1} (\cite[p.~720]{Tyszka})
Let \mbox{$D(x_1,\ldots,x_p) \in {\Z}[x_1,\ldots,x_p]$}.
Assume that \mbox{$d_i={\rm deg}(D,x_i) \geqslant 1$} for each \mbox{$i \in \{1,\ldots,p\}$}. We can compute a positive
integer \mbox{$n>p$} and a system \mbox{$T \subseteq E_n$} which satisfies the following three conditions:
\vskip 0.2truecm
\noindent
{\tt Condition 1.} If \mbox{$\K \in {\cal R}{\sl ng} \cup \{\N,\N \setminus \{0\}\}$}, then
\[
\forall \tilde{x}_1,\ldots,\tilde{x}_p \in \K ~\Bigl(D(\tilde{x}_1,\ldots,\tilde{x}_p)=0 \Longleftrightarrow
\]
\[
\exists \tilde{x}_{p+1},\ldots,\tilde{x}_n \in \K ~(\tilde{x}_1,\ldots,\tilde{x}_p,\tilde{x}_{p+1},\ldots,\tilde{x}_n) ~solves~ T\Bigr)
\]
{\tt Condition 2.} If \mbox{$\K \in {\cal R}{\sl ng} \cup \{\N,\N \setminus \{0\}\}$}, then
for each \mbox{$\tilde{x}_1,\ldots,\tilde{x}_p \in \K$} with \mbox{$D(\tilde{x}_1,\ldots,\tilde{x}_p)=0$},
there exists a unique tuple \mbox{$(\tilde{x}_{p+1},\ldots,\tilde{x}_n) \in {\K}^{n-p}$} such that the tuple
\mbox{$(\tilde{x}_1,\ldots,\tilde{x}_p,\tilde{x}_{p+1},\ldots,\tilde{x}_n)$} solves $T$.
\vskip 0.2truecm
\noindent
{\tt Condition 3.} If $M$ denotes the maximum of the absolute values of the
coefficients of \mbox{$D(x_1,\ldots,x_p)$}, then
\[
n=(M+2)^{\textstyle (d_1+1) \cdot \ldots \cdot (d_p+1)}-1
\]
Conditions 1 and 2 imply that for each \mbox{$\K \in {\cal R}{\sl ng} \cup \{\N,\N \setminus \{0\}\}$},
the equation \mbox{$D(x_1,\ldots,x_p)=0$} and the system $T$ have the same number of solutions in $\K$.
\end{lemma}
\begin{lemma}\label{lem2} (\cite[p.~100]{Robinson}) If $\LL \in {\cal R} \cup \{\N,\N \setminus \{0\}\}$
and \mbox{$x,y,z \in \LL$}, then \mbox{$z(x+y-z)=0$} if and only if
\[
(zx+1)(zy+1)=z^2(xy+1)+1
\]
\end{lemma}
\par
Let $\alpha$, $\beta$, and $\gamma$ denote variables.
\begin{lemma}\label{lem3}
If \mbox{$\LL \in {\cal R} \cup \{\N,\N\setminus\{0\}\}$} and \mbox{$x,y,z \in \LL$}, then \mbox{$x+y=z$} if and only if
\begin{equation}\label{equ1}
(zx+1)(zy+1)=z^2(xy+1)+1
\end{equation}
and
\begin{equation}\label{equ2}
((z+1)x+1)((z+1)(y+1)+1)=(z+1)^2(x(y+1)+1)+1
\end{equation}
We can express equations~(\ref{equ1}) and~(\ref{equ2}) as a system ${\cal F}$
such that ${\cal F}$ involves $x,y,z$ and $20$ new variables and ${\cal F}$ consists of equations
of the forms \mbox{$\alpha+1=\gamma$} and \mbox{$\alpha \cdot \beta=\gamma$}.
\end{lemma}
\begin{proof}
By Lemma~\ref{lem2}, equation~(\ref{equ1}) is equivalent to
\begin{equation}\label{equ3}
z(x+y-z)=0
\end{equation}
and equation~(\ref{equ2}) is equivalent to
\begin{equation}\label{equ4}
(z+1)(x+(y+1)-(z+1))=0
\end{equation}
\vskip 0.01truecm
\noindent
The conjunction of equations~(\ref{equ3}) and~(\ref{equ4}) is equivalent to \mbox{$x+y=z$}.
The new $20$ variables express the following $20$ polynomials:
\vskip 0.2truecm
\centerline{$zx$,~~~~~$zx+1$,~~~~~$zy$,~~~~~$zy+1$,~~~~~$z^2$,~~~~~$xy$,~~~~~$xy+1$,}
\vskip 0.2truecm
\centerline{$z^2(xy+1)$,~~~~~$z^2(xy+1)+1$,~~~~~$z+1$,~~~~~$(z+1)x$,}
\vskip 0.2truecm
\centerline{$(z+1)x+1$,~~~~~$y+1$,~~~~~$(z+1)(y+1)$,~~~~~$(z+1)(y+1)+1$,}
\vskip 0.2truecm
\centerline{$(z+1)^2$,~~~$x(y+1)$,~~~~~$x(y+1)+1$,}
\vskip 0.2truecm
\centerline{$(z+1)^2(x(y+1)+1)$,~~~~~$(z+1)^2(x(y+1)+1)+1$.}
\end{proof}
\begin{lemma}\label{lem4} (cf. Observation~\ref{3maj})
Let \mbox{$D(x_1,\ldots,x_p) \in$} \mbox{${\Z}[x_1,\ldots,x_p]$}.
Assume that \mbox{${\rm deg}(D,x_i) \geqslant 1$} for each \mbox{$i \in \{1,\ldots,p\}$}. We can compute a positive
integer \mbox{$n>p$} and a system \mbox{$T \subseteq G_n$} which satisfies the following two conditions:
\vskip 0.2truecm
\noindent
{\tt Condition 4.} If \mbox{$\K \in {\cal R}{\sl ng} \cup \{\N,\N \setminus \{0\}\}$}, then
\[
\forall \tilde{x}_1,\ldots,\tilde{x}_p \in \K ~\Bigl(D(\tilde{x}_1,\ldots,\tilde{x}_p)=0 \Longleftrightarrow
\]
\[
\exists \tilde{x}_{p+1},\ldots,\tilde{x}_n \in \K ~(\tilde{x}_1,\ldots,\tilde{x}_p,\tilde{x}_{p+1},\ldots,\tilde{x}_n) ~solves~ T\Bigr)
\]
{\tt Condition 5.} If \mbox{$\K \in {\cal R}{\sl ng} \cup \{\N,\N \setminus \{0\}\}$}, then
for each \mbox{$\tilde{x}_1,\ldots,\tilde{x}_p \in \K$} with \mbox{$D(\tilde{x}_1,\ldots,\tilde{x}_p)=0$},
there exists a unique tuple \mbox{$(\tilde{x}_{p+1},\ldots,\tilde{x}_n) \in {\K}^{n-p}$} such that the tuple
\mbox{$(\tilde{x}_1,\ldots,\tilde{x}_p,\tilde{x}_{p+1},\ldots,\tilde{x}_n)$} solves $T$.
\vskip 0.2truecm
\noindent
Conditions 4 and 5 imply that for each \mbox{$\K \in {\cal R}{\sl ng} \cup \{\N,\N \setminus \{0\}\}$},
the equation \mbox{$D(x_1,\ldots,x_p)=0$} and the system $T$ have the same number of solutions in $\K$.
\end{lemma}
\begin{proof}
Let the system \mbox{$T \subseteq E_n$} be given by Lemma~\ref{lem1}.
For every \mbox{$\LL \in {\cal R} \cup \{\N,\N \setminus \{0\}\}$},
\[
\forall x \in \LL ~\Bigl(x=1 \Longleftrightarrow \Bigl(x \cdot x=x \wedge x \cdot (x+1)=x+1\Bigr)\Bigr)
\]
Therefore, if there exists \mbox{$m \in \{1,\ldots,n\}$} such that
the equation \mbox{$1=x_m$} belongs to $T$, then we introduce a new variable $y$ and
replace in $T$ each equation of the form \mbox{$1=x_k$} by the equations
\mbox{$x_k \cdot x_k=x_k$}, \mbox{$x_k+1=y$}, \mbox{$x_k \cdot y=y$}.
Next, we apply Lemma~\ref{lem3} to each equation of the form \mbox{$x_i+x_j=x_k$} that belongs to $T$
and replace in $T$ each such equation by an equivalent system of equations of the forms
\mbox{$\alpha+1=\gamma$} and \mbox{$\alpha \cdot \beta=\gamma$}.
\end{proof}
\section{{\bf The main consequence of Conjecture~\ref{con1}}}
\begin{theorem}\label{the0}
Conjecture~\ref{con1} implies that there is an algorithm which takes as input a Diophantine equation,
returns an integer, and this integer is greater than the heights of rational solutions, if the solution set is finite.
\end{theorem}
\begin{proof}
It follows from Lemma~\ref{lem4} for \mbox{$\K=\Q$}.
The claim of Theorem~\ref{the0} also follows from Observation~\ref{3maj}.
\end{proof}
\begin{corollary}
Conjecture~\ref{con1} implies that the set of all Diophantine equations
which have infinitely many rational solutions is recursively enumerable.
Assuming Conjecture~\ref{con1}, a single query to the halting oracle
decides whether or not a given Diophantine equation has infinitely many rational solutions.
By the Davis-Putnam-Robinson-Matiyasevich theorem, the same is true for
an oracle that decides whether or not a given Diophantine equation has an integer solution.
\end{corollary}
\par
For many Diophantine equations we know that the number of rational solutions is
finite by Faltings' theorem. Faltings' theorem tells that certain curves
have finitely many rational points, but no known proof gives any bound on the sizes of the
numerators and denominators of the coordinates of those points, see \cite[p.~722]{Gowers}.
In all such cases Conjecture~\ref{con1} allows us to compute such a bound.
If this bound is small enough, that allows us to find all rational solutions by an exhaustive search.
For example, the equation \mbox{$x_1^5-x_1=x_2^2-x_2$} has only finitely
many rational solutions~\mbox{(\cite[p.~212]{Mignotte})}. The known rational solutions are:
\mbox{$(-1,0)$},
\mbox{$(-1,1)$},
\mbox{$(0,0)$},
\mbox{$(0,1)$},
\mbox{$(1,0)$},
\mbox{$(1,1)$},
\mbox{$(2,-5)$},
\mbox{$(2,6)$},
\mbox{$(3,-15)$},
\mbox{$(3,16)$},
\mbox{$(30,-4929)$},
\mbox{$(30,4930)$},
\mbox{$\left(\frac{1}{4},\frac{15}{32}\right)$},
\mbox{$\left(\frac{1}{4},\frac{17}{32}\right)$},
\mbox{$\left(-\frac{15}{16},-\frac{185}{1024}\right)$},
\mbox{$\left(-\frac{15}{16},\frac{1209}{1024}\right)$},
and the existence of other solutions is an open question, see \mbox{\cite[pp.~223--224]{Siksek}}. 
The system
\begin{displaymath}
\left\{
\begin{array}{rcl}
x_3+1 &=& x_2 \\
x_2 \cdot x_3 &=& x_4 \\
x_5+1 &=& x_1 \\
x_1 \cdot x_1 &=& x_6 \\
x_6 \cdot x_6 &=& x_7 \\
x_7 \cdot x_5 &=& x_4
\end{array}
\right.
\end{displaymath}
\vskip 0.01truecm
\noindent
is equivalent to \mbox{$x_1^5-x_1=x_2^2-x_2$}. By Conjecture~\ref{con1},
$h\left(x_1^4\right)=h\left(x_7\right) \leqslant h(x_1,\ldots,x_7) \leqslant 2^{\textstyle 2^{7-2}}=2^{32}$.
Therefore, $h(x_1) \leqslant \left(2^{32}\right)^{\textstyle \frac{1}{4}}=256$.
Assuming that Conjecture~\ref{con1} holds, the following {\sl MuPAD}
code finds all rational solutions of the equation \mbox{$x_1^5-x_1=x_2^2-x_2$}.
\begin{quote}
\begin{code}
solutions:={}:
for i from -256 to 256 do
for j from 1 to 256 do
x:=i/j:
y:=4*x^5-4*x+1:
p:=numer(y):
q:=denom(y):
if numlib::issqr(p) and numlib::issqr(q) then
z1:=sqrt(p/q):
z2:=-sqrt(p/q):
y1:=(z1+1)/2:
y2:=(z2+1)/2:
solutions:=solutions union {[x,y1],[x,y2]}:
end_if:
end_for:
end_for:
print(solutions):
\end{code}
\end{quote}
\vskip 0.01truecm
\noindent
The code solves the equivalent equation
\[
4x_1^5-4x_1+1=(2x_2-1)^2
\]
and displays the already presented solutions.
\vskip 0.2truecm
\par
{\sl MuPAD} is a general-purpose computer algebra system.
The commercial version of {\sl MuPAD} is no longer available as a \mbox{stand-alone}
product, but only as the {\sl Symbolic Math Toolbox} of {\sl MATLAB}. Fortunately,
this code can be executed by {\sl MuPAD Light},
which was offered for free for research and education until autumn 2005.
\section{{\bf Algebraic lemmas -- part 2}}
\begin{lemma}\label{march}
Lemmas~\ref{lem2} and~\ref{lem3} are not necessary for proving that in the rational
domain each Diophantine equation is equivalent to a system of equations of the forms
\mbox{$\alpha+1=\gamma$} and \mbox{$\alpha \cdot \beta=\gamma$}.
\end{lemma}
\begin{proof}
By Lemma~\ref{lem1}, an arbitrary Diophantine equation is equivalent to a system \mbox{$T \subseteq E_n$},
where $n$ and $T$ can be computed. If there exists \mbox{$m \in \{1,\ldots,n\}$}
such that the equation \mbox{$1=x_m$} belongs to $T$, then we introduce a new variable $t$ and
replace in $T$ each equation of the form \mbox{$1=x_k$} by the equations
\mbox{$x_k \cdot x_k=x_k$}, \mbox{$x_k+1=t$}, and \mbox{$x_k \cdot t=t$}.
For each rational number~$y$, we have \mbox{$y^2+1 \neq 0$} and \mbox{$y(y^2+1)+1 \neq 0$}. Hence,
for each rational numbers $x$, $y$, $z$,
\[
x+y=z ~ \Longleftrightarrow ~ x(y^2+1)+y(y^2+1)=z(y^2+1) ~ \Longleftrightarrow
\]
\[
x(y^2+1)+y(y^2+1)+1=z(y^2+1)+1 ~ \Longleftrightarrow
\]
\[
\left(y(y^2+1)+1\right) \cdot \left(\frac{x(y^2+1)}{y(y^2+1)+1}+1\right)=z(y^2+1)+1
\]
We transform the last equation into an equivalent system \mbox{$W \subseteq G_{12}$} in such
a way that the variables \mbox{$x_1,\ldots,x_{12}$} correspond to the following rational expressions:
\[
x,~y,~z,~y^2,~y^2+1,~y(y^2+1),~y(y^2+1)+1,~x(y^2+1),
\]
\[
\frac{x(y^2+1)}{y(y^2+1)+1},~\frac{x(y^2+1)}{y(y^2+1)+1}+1,~z(y^2+1),~z(y^2+1)+1.
\]
In this way, we replace in $T$ each equation of the form \mbox{$x_i+x_j=x_k$} by an equivalent
system of equations of the forms \mbox{$\alpha+1=\gamma$} and \mbox{$\alpha \cdot \beta=\gamma$}.
\end{proof}
\par
The next lemma enable us to prove Theorem~\ref{the0} without using Lemma~\ref{lem4}.
\begin{lemma}\label{april}
For solutions in a field, each system \mbox{${\cal S} \subseteq E_n$}
is equivalent to \mbox{$T_1 \vee \cdots \vee T_p$}, where each $T_i$
is a system of equations of the forms \mbox{$\alpha+1=\gamma$}
and \mbox{$\alpha \cdot \beta=\gamma$}.
\end{lemma}
\begin{proof}
Acting as in the proof of Lemma~\ref{march}, we eliminate from~${\cal S}$ all equations of the form \mbox{$1=x_k$}.
Let $m$ denote the number of equations of the form \mbox{$x_i+x_j=x_k$}
that belong to~${\cal S}$. We can assume that \mbox{$m>0$}. Let the variables
$y$, $z$, $t$, $w$, $s$, and $r$ be new. Let
\[
{\cal S}_1=\left({\cal S} \setminus \{x_i+x_j=x_k\}\right) \cup
\]
\[
\{x_i+1=y,~~x_k+1=y,~~x_j+1=z,~~z \cdot x_j=x_j\}
\]
and let
\[
{\cal S}_2=\left({\cal S} \setminus \{x_i+x_j=x_k\}\right) \cup
\]
\[
\{t \cdot x_j=x_i,~~t+1=w,~~w \cdot x_j=x_k,~~x_j+1=s,~~r \cdot x_j=s\}
\]
The system ${\cal S}_1$ expresses that \mbox{$x_i+x_j=x_k$} and \mbox{$x_j=0$}.
The system ${\cal S}_2$ expresses that \mbox{$x_i+x_j=x_k$} and \mbox{$x_j \neq 0$}.
Therefore, \mbox{${\cal S} \Longleftrightarrow \left({\cal S}_1 \vee {\cal S}_2\right)$}.
We have described a procedure which transforms ${\cal S}$ into ${\cal S}_1$ and ${\cal S}_2$.
We iterate this procedure for ${\cal S}_1$ and ${\cal S}_2$
and finally obtain the systems \mbox{$T_1,\ldots,T_{2^m}$} without equations of the form
\mbox{$x_i+x_j=x_k$}. The systems \mbox{$T_1,\ldots,T_{2^m}$}
satisfy \mbox{${\cal S} \Longleftrightarrow \left(T_1 \vee \cdots \vee T_{2^m}\right)$} and they contain only
equations of the forms \mbox{$\alpha+1=\gamma$} and \mbox{$\alpha \cdot \beta=\gamma$}.
\end{proof}
\section{{\bf Systems which have infinitely many rational solutions}}
\begin{lemma}\label{sierpinski} (\cite[p.~391]{Sierpinski})
If $2$ has an odd exponent in the prime factorization of a positive \mbox{integer $n$},
then $n$ can be written as the sum of three squares of integers.
\end{lemma}
\begin{lemma}\label{or}
For each positive rational number $z$, $z$ or $2z$ can be written as the sum of three squares of rational numbers.
\end{lemma}
\begin{proof}
We find positive integers $p$ and $q$ with \mbox{$z=\frac{p}{q}$}.
If $2$ has an odd exponent in the prime factorization of $pq$, then by
Lemma~\ref{sierpinski} there exist integers $i_1$, $i_2$, $i_3$ such that
\mbox{$pq=i_1^2+i_2^2+i_3^2$}. Hence,
\[
z=\left(\frac{i_1}{q}\right)^2+\left(\frac{i_2}{q}\right)^2+\left(\frac{i_3}{q}\right)^2
\]
If $2$ has an even exponent in the prime factorization of $pq$, then
by Lemma~\ref{sierpinski} there exist integers $j_1$, $j_2$, $j_3$ such that \mbox{$2pq=j_1^2+j_2^2+j_3^2$}. Hence,
\[
2z=\left(\frac{j_1}{q}\right)^2+\left(\frac{j_2}{q}\right)^2+\left(\frac{j_3}{q}\right)^2
\]
\end{proof}
\begin{lemma}\label{technical}
A rational number $z$ can be written as the sum of three squares of rational numbers if and only if
there exist rational numbers $r$, $s$, $t$ such that \mbox{$z=r^2\left(s^2\left(t^2+1\right)+1\right)$}.
\end{lemma}
\begin{proof}
Let \mbox{$H(r,s,t)=r^2\left(s^2\left(t^2+1\right)+1\right)$}. Of course,
\[
H(r,s,t)=r^2+(rs)^2+(rst)^2
\]
We prove that for each rational numbers $a$, $b$, $c$ there exist rational numbers $r$, $s$, $t$ such
that \mbox{$a^2+b^2+c^2=H(r,s,t)$}. Without loss of generality we can assume that
\mbox{$|a| \leqslant |b| \leqslant |c|$}. If \mbox{$b=0$}, then \mbox{$a=0$} and \mbox{$a^2+b^2+c^2=H(c,0,0)$}.
If \mbox{$b \neq 0$}, then \mbox{$c \neq 0$} and \mbox{$a^2+b^2+c^2=H\left(c,\frac{b}{c},\frac{a}{b}\right)$}.
\end{proof}
\begin{lemma}\label{bremner} (\cite[p.~125]{Bremner})
The equation \mbox{$x^3+y^3=4981$} has infinitely many solutions in positive rationals and each
such solution \mbox{$(x,y)$} satisfies \mbox{$h(x,y)>10^{\textstyle 16 \cdot 10^{6}}$}.
\end{lemma}
\begin{theorem}\label{april1}
There exists a system \mbox{${\cal T} \subseteq G_{28}$} such that ${\cal T}$
has infinitely many solutions in rationals \mbox{$x_1,\ldots,x_{28}$} and each
such solution \mbox{$(x_1,\ldots,x_{28})$} has height greater than \mbox{$2^{\textstyle 2^{27}}$}.
\end{theorem}
\begin{proof}
We define:
\[
\Omega=\Bigl\{\rho \in \Q \cap (0,\infty): \exists y \in \Q ~\left(\rho \cdot y\right)^3+y^3=4981\Bigr\}
\]
Let $\Omega_1$ denote the set of all positive rationals $\rho$ such that the system
\begin{displaymath}
\left\{
\begin{array}{rcl}
\left(\rho \cdot y \right)^3+y^3&=& 4981 \\
\rho^3 &=& a^2+b^2+c^2
\end{array}
\right.
\end{displaymath}
is solvable in rationals. Let $\Omega_2$ denote the set of all positive rationals $\rho$ such that the system
\begin{displaymath}
\left\{
\begin{array}{rcl}
\left(\rho \cdot\ y \right)^3+y^3&=& 4981 \\
2\rho^3 &=& a^2+b^2+c^2
\end{array}
\right.
\end{displaymath}
is solvable in rationals. Lemma~\ref{bremner} implies that the set $\Omega$ is infinite.
By Lemma~\ref{or}, \mbox{$\Omega=\Omega_1 \cup \Omega_2$}. Therefore,
$\Omega_1$ is infinite (Case 1) or $\Omega_2$ is infinite (Case 2).
\par
Case 1. In this case the system
\begin{displaymath}
\left\{
\begin{array}{rcl}
x^3+y^3&=& 4981 \\
\frac{x^3}{y^3} &=& a^2+b^2+c^2
\end{array}
\right.
\end{displaymath}
has infinitely many rational solutions. By this
and Lemma~\ref{technical}, the system
\begin{displaymath}
\left\{
\begin{array}{rcl}
x^3+y^3 &=& 4981 \\
\frac{x^3}{y^3} &=& r^2\left(s^2\left(t^2+1\right)+1\right)
\end{array}
\right.
\end{displaymath}
has infinitely many rational solutions.
We transform the above system into an equivalent system \mbox{${\cal T} \subseteq G_{27}$} in such a way
that the variables \mbox{$x_1,\cdots,x_{27}$} correspond to the following rational
expressions:
\vskip 0.2truecm
\centerline{$x$,~~$y$,~~$x^2$,~~$x^3$,~~$y^2$,~~$y^3$,~~$\frac{x^3}{y^3}$,~~$\frac{x^3}{y^3}+1$,}
\vskip 0.2truecm
\centerline{$1$,~~$2$,~~$4$,~~$16$,~~$17$,~~$289$,~~$\frac{289}{4}$,~~$\frac{289}{4}+1$,~~$293$,~~$4981$,}
\vskip 0.2truecm
\centerline{$t$,~~$t^2$,~~$t^2+1$,~~$s$,~~$s^2$,~~$s^2(t^2+1)$,~~$s^2(t^2+1)+1$,~~$r$,~~$r^2$.}
\vskip 0.2truecm
\noindent
The system ${\cal T}$ has infinitely many solutions in rationals \mbox{$x_1,\ldots,x_{27}$}.
Lemma~\ref{bremner} implies that each rational tuple \mbox{$(x_1,\ldots,x_{27})$} that solves ${\cal T}$ satisfies
\[
h\left(x_{1},\ldots,x_{27}\right) \geqslant h\left(x_{1}^{3}, x_{2}^{3}\right)=
\Big(h\left(x_{1},x_{2}\right)\!\Big)^{3}>10^{\textstyle 48 \cdot 10^{6}} > 2^{\textstyle 2^{27}}
\]
Since \mbox{$G_{27} \subseteq G_{28}$}, \mbox{${\cal T} \subseteq G_{28}$} and the proof for Case 1 is complete.
\par
Case 2. In this case the system
\begin{displaymath}
\left\{
\begin{array}{rcl}
x^3+y^3&=& 4981 \\
2 \cdot \frac{x^3}{y^3} &=& a^2+b^2+c^2
\end{array}
\right.
\end{displaymath}
has infinitely many rational solutions. By this
and Lemma~\ref{technical}, the system
\begin{displaymath}
\left\{
\begin{array}{rcl}
x^3+y^3 &=& 4981 \\
2 \cdot \frac{x^3}{y^3} &=& r^2\left(s^2\left(t^2+1\right)+1\right)
\end{array}
\right.
\end{displaymath}
has infinitely many rational solutions.
We transform the above system into an equivalent system \mbox{${\cal T} \subseteq G_{28}$} in such a way
that the variables \mbox{$x_{1},\ldots,x_{28}$} correspond to the following rational expressions:
\vskip 0.2truecm
\centerline{$x$,~$y$,~$x^2$,~$x^3$,~$y^2$,~$y^3$,~$\frac{x^3}{y^3}$,~$2 \cdot\frac{x^3}{y^3}$,~$\frac{x^3}{y^3}+1$,}
\vskip 0.2truecm
\centerline{$1$,~$2$,~$4$,~$16$,~$17$,~$289$,~$\frac{289}{4}$,~$\frac{289}{4}+1$,~$293$,~$4981$,}
\vskip 0.2truecm
\centerline{$t$,~$t^2$,~$t^2+1$,~$s$,~$s^2$,~$s^2(t^2+1)$,~$s^2(t^2+1)+1$,~$r$,~$r^2$.}
\vskip 0.2truecm
\noindent
The system ${\cal T}$ has infinitely many solutions in rationals \mbox{$x_1,\ldots,x_{28}$}.
Lemma~\ref{bremner} implies that each rational tuple \mbox{$(x_1,\ldots,x_{28})$} that solves ${\cal T}$ satisfies
\[
h\left(x_{1},\ldots,x_{28}\right) \geqslant h\left(x_{1}^{3}, x_{2}^{3}\right)=
\Big(h\left(x_{1},x_{2}\right)\!\Big)^{3}>10^{\textstyle 48 \cdot 10^{6}} > 2^{\textstyle 2^{27}}
\]
\end{proof}
\par
For a positive integer $n$, let \mbox{$\mu(n)$} denote the smallest positive integer $m$
such that each system \mbox{${\cal S} \subseteq G_n$} solvable in rationals \mbox{$x_1,\ldots,x_n$}
has a rational solution \mbox{$(x_1,\ldots,x_n)$} whose height is not greater than $m$.
Obviously, \mbox{$\mu(1)=1$}. Observation~\ref{obs0} implies that \mbox{$\mu(n) \geqslant 2^{\textstyle 2^{n-2}}$}
for every integer \mbox{$n \geqslant 2$}. Theorem~\ref{april1} implies that \mbox{$\mu(28)>2^{\textstyle 2^{27}}$}.
\begin{theorem}
The function \mbox{$\mu \colon \N \setminus \{0\} \to \N \setminus \{0\}$} is computable in the limit.
\end{theorem}
\begin{proof}
Let us agree that the empty tuple has height $0$.
For a positive integer $w$ and a tuple
\[
\left(x_{1},\ldots,x_{n}\right) \in \left([-w,w] \cap \Z\right)^{n} \setminus \big\{(\underbrace{w,\ldots,w}_{n-\text{times}})\big\}
\]
let \mbox{${\rm succ}\left(\left(x_{1},\ldots,x_{n}\right),w\right)$} denote the successor of \mbox{$\left(x_{1},\ldots,x_{n}\right)$}
in the co-lexicographic order on \mbox{$\left([-w,w] \cap \Z\right)^{n}$}.
Flowchart 2 illustrates an infinite-time computation of \mbox{$\mu(n)$}.
\begin{center}
\includegraphics[width=\hsize]{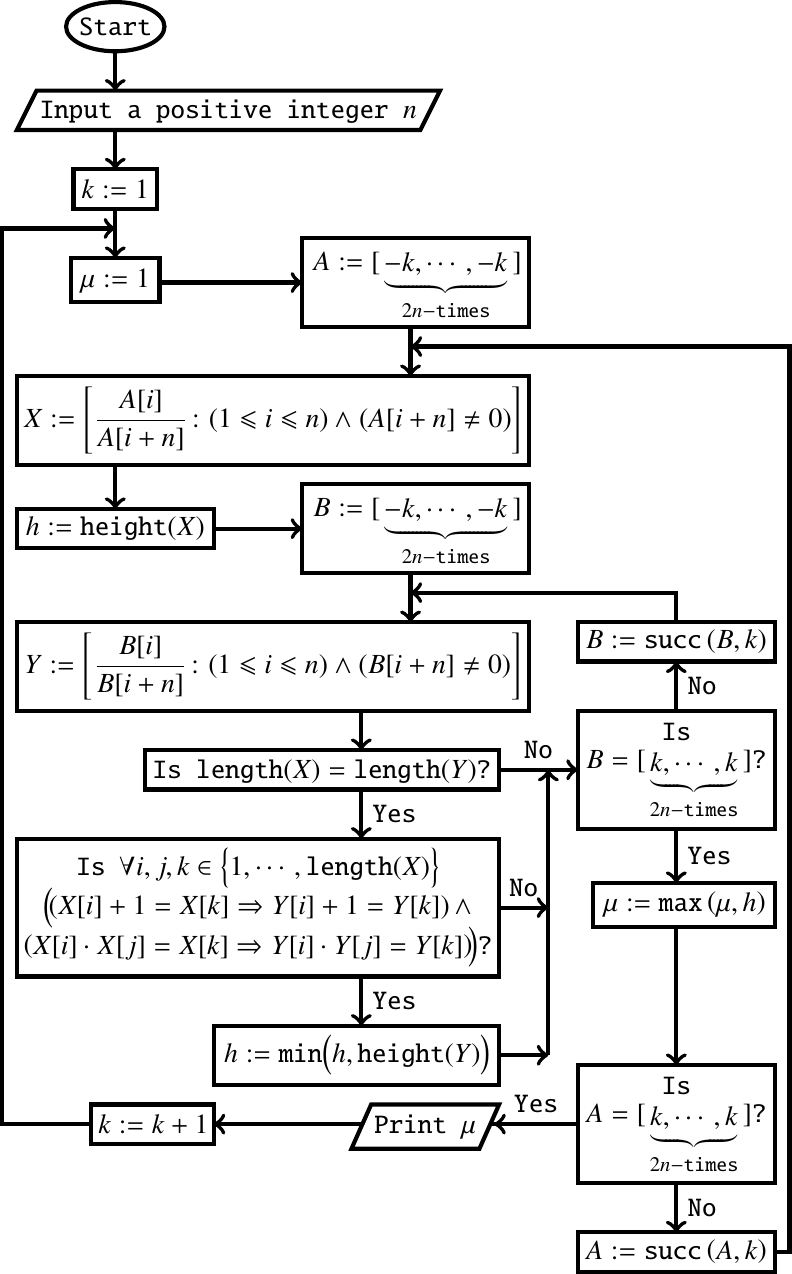}
\end{center}
\vskip 0.01truecm
\centerline{{\it Flowchart 2: An \mbox{infinite-time} computation of $\mu(n)$}}
\end{proof}
\section{{\bf Conjecture~\ref{con2} and its equivalent form}}
Let $[\cdot]$ denote the integer part function.
\begin{lemma}\label{lem22}
For every \mbox{non-negative} real numbers $x$ and $y$, \mbox{$x+1=y$} implies that
\mbox{$2^{\textstyle 2^{[x]}} \cdot 2^{\textstyle 2^{[x]}}=2^{\textstyle 2^{[y]}}$}.
\end{lemma}
\begin{proof}
For every \mbox{non-negative} real numbers $x$ and $y$, \mbox{$x+1=y$} implies that \mbox{$[x]+1=[y]$}.
\end{proof}
\par
Let \mbox{$f(1)=1$}, and let \mbox{$f(n+1)=2^{\textstyle 2^{\textstyle f(n)}}$} for every positive integer $n$.
Let \mbox{$g(1)=0$}, and let \mbox{$g(n+1)=2^{\textstyle 2^{\textstyle g(n)}}$} for every positive integer $n$.
\begin{conjecture}\label{con2}
If a system \mbox{${\cal S} \subseteq G_n$}
has only finitely many solutions in \mbox{non-negative} rationals \mbox{$x_1,\ldots,x_n$},
then each such solution \mbox{$(x_1,\ldots,x_n)$} satisfies \mbox{$h(x_1,\ldots,x_n) \leqslant f(2n)$}.
\end{conjecture}
\par
Observations~\ref{obs3} and \ref{obs4} justify Conjecture~\ref{con2}.
\begin{observation}\label{obs3}
For every system \mbox{${\cal S} \subseteq G_n$} which involves all the variables \mbox{$x_1,\ldots,x_n$},
the following new system
\[
{\cal S} \cup \left\{2^{\textstyle 2^{\textstyle [x_k]}}=y_k:~k \in \{1,\ldots,n\}\right\} \cup
\bigcup_{x_i+1=x_k \in {\cal S}} \left\{y_i \cdot y_i=y_k\right\}
\]
is equivalent to ${\cal S}$. If the system ${\cal S}$ has only finitely many solutions in \mbox{non-negative}
rationals \mbox{$x_1,\ldots,x_n$}, then the new system has only finitely many solutions in \mbox{non-negative}
rationals \mbox{$x_1,\ldots,x_n,y_1,\ldots,y_n$}.
\end{observation}
\begin{proof}
It follows from Lemma~\ref{lem22}.
\end{proof}
\begin{observation}\label{obs4}
For every positive integer $n$, the following system
\[
\left\{\begin{array}{rcl}
x_1 \cdot x_1 &=& x_1 \\
\forall i \in \{1,\ldots,n-1\} ~2^{\textstyle 2^{\textstyle [x_i]}} &=& x_{i+1} ~({\rm if~} n>1)
\end{array}\right.
\]
has exactly two solutions in \mbox{non-negative} rationals,
namely \mbox{$(g(1),\ldots,g(n))$}
and \mbox{$(f(1),\ldots,f(n))$}. The second solution has greater height.
\end{observation}
\par
Conjecture~\ref{con2} is equivalent to the following conjecture on the
arithmetic of \mbox{non-negative} rationals:
if \mbox{non-negative} rational numbers \mbox{$x_1,\ldots,x_n$} satisfy \mbox{$h(x_1,\ldots,x_n)>f(2n)$},
then there exist \mbox{non-negative} rational numbers \mbox{$y_1,\ldots,y_n$} such that
\[
h(x_1,\ldots,x_n)<h(y_1,\ldots,y_n)
\]
and for every \mbox{$i,j,k \in \{1,\ldots,n\}$}
\[
(x_i+1=x_k \Longrightarrow y_i+1=y_k) ~\wedge (x_i \cdot x_j=x_k \Longrightarrow y_i \cdot y_j=y_k)
\]
\begin{theorem}\label{fracja}
Conjecture~\ref{con2} is true if and only if the execution of Flowchart 3 prints infinitely many numbers.
\end{theorem}
\begin{center}
\includegraphics[width=\hsize]{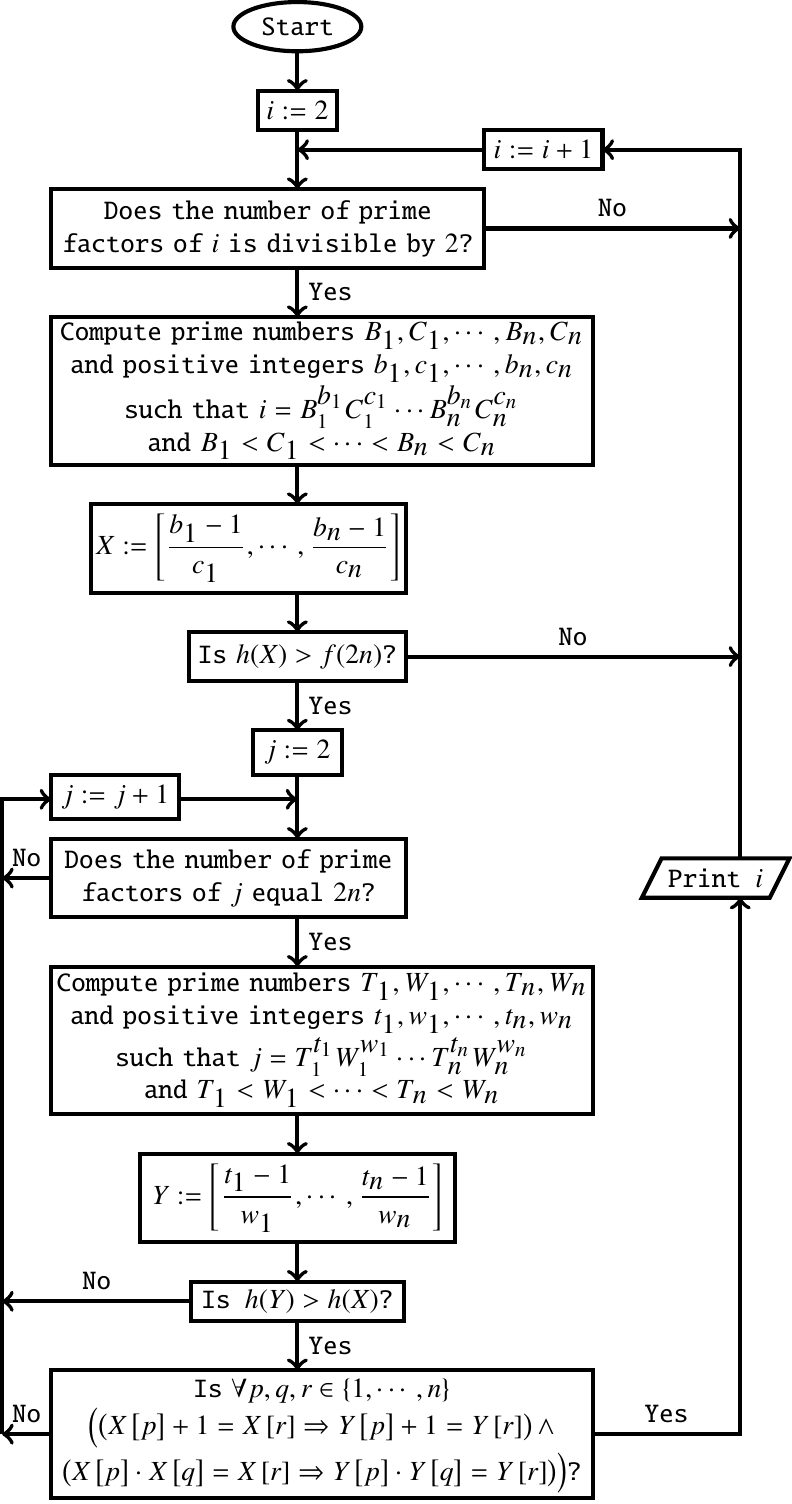}
\end{center}
\vskip 0.01truecm
\centerline{{\it Flowchart 3: An infinite-time computation which}}
\vskip 0.01truecm
\centerline{{\it decides whether or not Conjecture~\ref{con2} is true}}
\begin{proof}
Let $\Gamma_2$ denote the set of all integers \mbox{$i \geqslant 2$} whose number of prime factors is
divisible by $2$. The claimed equivalence is true because the algorithm from Flowchart 3
applies a surjective function from $\Gamma_2$ to \mbox{$\bigcup_{n=1}^\infty\limits \left({\Q} \cap [0,\infty)\right)^n$}.
\end{proof}
\begin{corollary}
Conjecture~\ref{con2} can be written in the form
\mbox{$\forall x \in \N ~\exists y \in \N ~\psi(x,y)$},
where \mbox{$\psi(x,y)$} is a computable predicate.
\end{corollary}
\section{{\bf Algebraic lemmas -- part 3}}
\begin{lemma}\label{lem6} (cf. \cite[p.~100]{Robinson})
For every \mbox{non-negative} real numbers \mbox{$x,y,z$}, \mbox{$x+y=z$} if and only if
\begin{equation}\label{equ11}
((z+1)x+1)((z+1)(y+1)+1)=(z+1)^2(x(y+1)+1)+1
\end{equation}
\end{lemma}
\begin{proof}
The left side of equation~(\ref{equ11}) minus the right side of equation~(\ref{equ11}) equals \mbox{$(z+1)(x+y-z)$}.
\end{proof}
\begin{lemma}\label{lem7}
In \mbox{non-negative} rationals, the equation \mbox{$x+y=z$} is equivalent to a system
which consists of equations of the forms \mbox{$\alpha+1=\gamma$} and \mbox{$\alpha \cdot \beta=\gamma$}.
\end{lemma}
\begin{proof}
It follows from Lemma~\ref{lem6}.
\end{proof}
\begin{lemma}\label{lem9}
Let \mbox{$D(x_1,\ldots,x_p) \in {\Z}[x_1,\ldots,x_p]$}.
Assume that \mbox{${\rm deg}(D,x_i) \geqslant 1$} for each \mbox{$i \in \{1,\ldots,p\}$}. We can compute a positive
integer \mbox{$n>p$} and a system \mbox{${\cal T} \subseteq G_n$} which satisfies the following two conditions:
\vskip 0.2truecm
\noindent
{\tt Condition 6.} For every \mbox{non-negative} rationals \mbox{$\tilde{x}_1,\ldots,\tilde{x}_p$},
\[
D(\tilde{x}_1,\ldots,\tilde{x}_p)=0 \Longleftrightarrow
\]
\[
\exists \tilde{x}_{p+1},\ldots,\tilde{x}_n \in {\Q} \cap [0,\infty) ~(\tilde{x}_1,\ldots,\tilde{x}_p,\tilde{x}_{p+1},\ldots,\tilde{x}_n) ~solves~ {\cal T}
\]
{\tt Condition 7.} If \mbox{non-negative} rationals \mbox{$\tilde{x}_1,\ldots,\tilde{x}_p$} satisfy
\mbox{$D(\tilde{x}_1,\ldots,\tilde{x}_p)=0$}, then there exists a unique tuple
\mbox{$(\tilde{x}_{p+1},\ldots,\tilde{x}_n) \in {\left(\Q \cap [0,\infty)\right)}^{n-p}$} such that the tuple
\mbox{$(\tilde{x}_1,\ldots,\tilde{x}_p,\tilde{x}_{p+1},\ldots,\tilde{x}_n)$} solves~${\cal T}$.
\vskip 0.2truecm
\noindent
Conditions 6 and 7 imply that the equation \mbox{$D(x_1,\ldots,x_p)=0$} and the system ${\cal T}$ have
the same number of solutions in \mbox{non-negative} rationals.
\end{lemma}
\begin{proof}
We write down the polynomial \mbox{$D(x_1,\ldots,x_p)$} and replace each coefficient by the successor
of its absolute value. Let \mbox{$\widetilde{D}(x_1,\ldots,x_p)$} denote the obtained polynomial.
The polynomials \mbox{$D(x_1,\ldots,x_p)+\widetilde{D}(x_1,\ldots,x_p)$} and \mbox{$\widetilde{D}(x_1,\ldots,x_p)$}
have positive integer coefficients. The equation \mbox{$D(x_1,\ldots,x_p)=0$} is equivalent to
\[
D(x_1,\ldots,x_p)+\widetilde{D}(x_1,\ldots,x_p)+1=\widetilde{D}(x_1,\ldots,x_p)+1
\]
There exist a positive integer $a$ and a finite \mbox{non-empty} list $A$ such that
\[
D(x_1,\ldots,x_p)+\widetilde{D}(x_1,\ldots,x_p)+1=
\]
\begin{equation}\label{equ22}
\Bigl(\Bigl(\Bigl(\sum_{\textstyle (i_1,j_1,\ldots,i_k,j_k) \in A}
x_{\textstyle i_1}^{\textstyle j_1} \cdot \ldots \cdot x_{\textstyle i_k}^{\textstyle j_k}\Bigr)+
\underbrace{1\Bigr)+\ldots\Bigr)+1}_{\textstyle a~{\rm units}}
\end{equation}
and all the numbers \mbox{$k,i_1,j_1,\ldots,i_k,j_k$} belong to \mbox{$\N \setminus \{0\}$}.
There exist a positive integer $b$ and a finite \mbox{non-empty} list $B$ such that
\vskip 0.01truecm
\[
\widetilde{D}(x_1,\ldots,x_p)+1=
\]
\begin{equation}\label{equ33}
\Bigl(\Bigl(\Bigl(\sum_{\textstyle (i_1,j_1,\ldots,i_k,j_k) \in B}
x_{\textstyle i_1}^{\textstyle j_1} \cdot \ldots \cdot x_{\textstyle i_k}^{\textstyle j_k}\Bigr)+
\underbrace{1\Bigr)+\ldots\Bigr)+1}_{\textstyle b~{\rm units}}
\end{equation}
and all the numbers \mbox{$k,i_1,j_1,\ldots,i_k,j_k$} belong to \mbox{$\N \setminus \{0\}$}.
By Lemma~\ref{lem7}, we can equivalently express the equality of the right
sides of equations (\ref{equ22}) and (\ref{equ33}) using only equations of the forms
\mbox{$\alpha+1=\gamma$} and \mbox{$\alpha \cdot \beta=\gamma$}.
Consequently, we can effectively find the system ${\cal T}$.
\end{proof}
\begin{observation}\label{3maj}
Combining the above reasoning with Lemma~\ref{lem3} for \mbox{$\LL=\Q$}, we can prove Lemma~\ref{lem4} for \mbox{$\K=\Q$}.
\end{observation}
\section{{\bf Consequences of Conjecture~\ref{con2}}}
\begin{theorem}\label{the2}
If we assume Conjecture~\ref{con2} and a Diophantine equation \mbox{$D(x_1,\ldots,x_p)=0$} has only
finitely many solutions in \mbox{non-negative} rationals, then an upper bound for their heights can be computed.
\end{theorem}
\begin{proof}
It follows from Lemma~\ref{lem9}.
\end{proof}
\begin{theorem}\label{the5}
If we assume Conjecture~\ref{con2} and a Diophantine equation \mbox{$D(x_1,\ldots,x_p)=0$} has only finitely many
rational solutions, then an upper bound for their heights can be computed by applying Theorem~\ref{the2}
to the equation
\[
\prod\limits_{\textstyle (i_1,\ldots,i_p) \in \{1,2\}^p}
D((-1)^{\textstyle i_1} \cdot x_1,\ldots,(-1)^{\textstyle i_p} \cdot x_p)=0
\]
\end{theorem}
\begin{corollary}
Conjecture~\ref{con2} implies that the set of all Diophantine equations
which have infinitely many rational solutions is recursively enumerable.
Assuming Conjecture~\ref{con2}, a single query to the halting oracle
decides whether or not a given Diophantine equation has infinitely many rational solutions.
By the Davis-Putnam-Robinson-Matiyasevich theorem, the same is true for
an oracle that decides whether or not a given Diophantine equation has an integer solution.
\end{corollary}
\section{{\bf Theorems on relative decidability}}
\vskip 0.01truecm
\noindent
{\bf Question}~(\cite{MathOverflow}). {\em Can the twin prime problem be solved with a single use of a halting oracle?}
\vskip 0.2truecm
\par
Let \mbox{$\xi(3)=4$}, and let \mbox{$\xi(n+1)=\xi(n)!$} for every integer \mbox{$n \geqslant 3$}.
For an integer \mbox{$n \geqslant 3$}, let \mbox{$\Psi_n$} denote the statement: if a system
\mbox{${\cal S}~\subseteq~\Bigl\{x_i!=x_{i+1}:~1 \leqslant i \leqslant n-1\Bigr\} \cup$}
\mbox{$\Bigl\{x_i \cdot x_j=x_{j+1}:~1 \leqslant i \leqslant j \leqslant n-1\Bigr\}$}
has only finitely many solutions in positive integers \mbox{$x_1,\ldots,x_n$},
then each such solution \mbox{$(x_1,\ldots,x_n)$} satisfies \mbox{$x_1,\ldots,x_n \leqslant \xi(n)$}.
\begin{theorem}~(\cite{Tyszka1})
The statement \mbox{$\Psi_{16}$} proves the implication: if there exists a twin prime greater than \mbox{$\xi(14)$},
then there are infinitely many twin primes.
\end{theorem}
\begin{corollary}
Assuming the statement \mbox{$\Psi_{16}$}, a single query to the halting oracle decides the validity
of the twin prime conjecture.
\end{corollary}
\begin{conjecture}\label{con3} (Harvey Friedman's conjecture in \cite{Friedman})
The set of all Diophantine equations which have only finitely many rational solutions is not recursively enumerable.
\end{conjecture}
\par
Conjecture~\ref{con3} implies Conjecture~\ref{con4}.
\begin{conjecture}\label{con4}
The set of all Diophantine equations which have only finitely many rational solutions is not computable.
\end{conjecture}
By Theorem~\ref{the0}, Conjecture~\ref{con1} implies Conjecture~\ref{maj2}.
By Theorem~\ref{the5}, Conjecture~\ref{con2} implies Conjecture~\ref{maj2}.
\begin{conjecture}\label{maj2}
There is an algorithm which takes as input a Diophantine equation \mbox{$D(x_1,\ldots,x_p)=0$},
returns an integer \mbox{$b \geqslant 2$}, where $b$ is greater than the number of rational solutions,
if the solution set is finite.
\end{conjecture}
\vskip 0.01truecm
\noindent
{\bf Guess}~\mbox{(\cite[p.~16]{Kim})}. {\em The question whether or not a given Diophantine
equation has only finitely many rational solutions is decidable with an oracle that
decides whether or not a given Diophantine equation has a rational solution.}
\vskip 0.2truecm
\par
Originally, Minhyong Kim formulated the Guess as follows: for rational solutions, the finiteness
problem is decidable relative to the existence problem.
Conjecture~\ref{con4} and the Guess imply that there is no algorithm
which decides whether or not a Diophantine equation has a rational solution.
Martin Davis' conjecture in \mbox{\cite[p.~729]{Davis}} implies the same.
\begin{theorem}\label{the8}
Conjecture~\ref{maj2} implies
that the question whether or not a given Diophantine
equation has only finitely many rational solutions is decidable by a single query to an oracle
that decides whether or not a given Diophantine equation has a rational solution.
\end{theorem}
\begin{proof}
Assuming that Conjecture~\ref{maj2} holds, the execution of Flowchart~4 decides whether or not a Diophantine equation
\mbox{$D(x_1,\ldots,x_p)=0$} has only finitely many rational solutions.
\begin{center}
\includegraphics[width=\hsize]{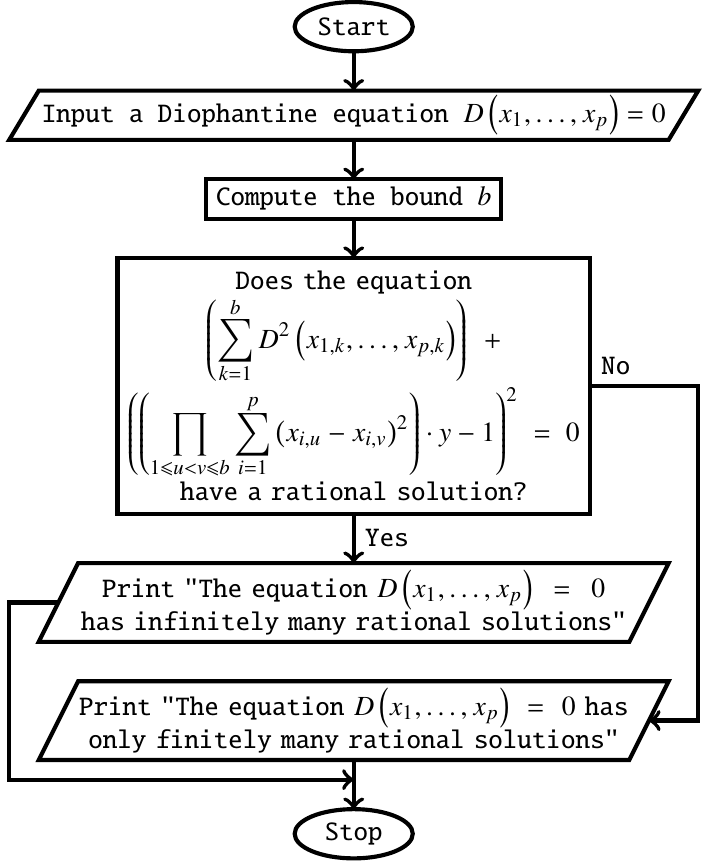}
\end{center}
\vskip 0.01truecm
\centerline{{\it Flowchart 4: Conjecture~\ref{maj2} implies the Guess}}
\end{proof}
\begin{corollary}
Conjecture~\ref{maj2} implies
that the question whether or not a given Diophantine
equation has only finitely many rational solutions is decidable by a single query to an oracle
that decides whether or not a given Diophantine equation has an integer solution.
\end{corollary}
\begin{lemma}\label{xyz}
A Diophantine equation \mbox{$D(x_1,\ldots,x_p)=0$} has no solutions in rationals \mbox{$x_1,\ldots,x_p$}
if and only if the equation \mbox{$D(x_1,\ldots,x_p)+0 \cdot x_{p+1}=0$} has only finitely many
solutions in rationals \mbox{$x_1,\ldots,x_{p+1}$}.
\end{lemma}
\begin{theorem}
If the set of all Diophantine equations which
have only finitely many rational solutions is recursively enumerable, then there exists an algorithm
which decides whether or not a given Diophantine equation has a rational solution.
\end{theorem}
\begin{proof}
For a \mbox{non-negative} integer $n$, we define
\[
\theta(n)=\left\{
\begin{array}{lcl}
\eta(n+2) && {\rm (if~} n+2 \in \Gamma_3 {\rm )}\\
0 && {\rm (if~} n+2 \not\in \Gamma_3 {\rm )}
\end{array}
\right.
\]
where $\eta$ and $\Gamma_3$ were defined in the proof of Theorem~\ref{new2}.
The function \mbox{$\theta \colon \N \to \bigcup_{n=1}^\infty\limits {\Q}^n$}
is computable and surjective.
Suppose that \mbox{$\{{\cal A}_n=0\}_{n=0}^\infty$} is a computable sequence of all Diophantine
equations which have only finitely many rational solutions.
By Lemma~\ref{xyz}, the execution of Flowchart 5 decides whether or not a Diophantine
equation \mbox{$D(x_1,\ldots,x_p)=0$} has a rational solution.
\begin{center}
\includegraphics[width=\hsize]{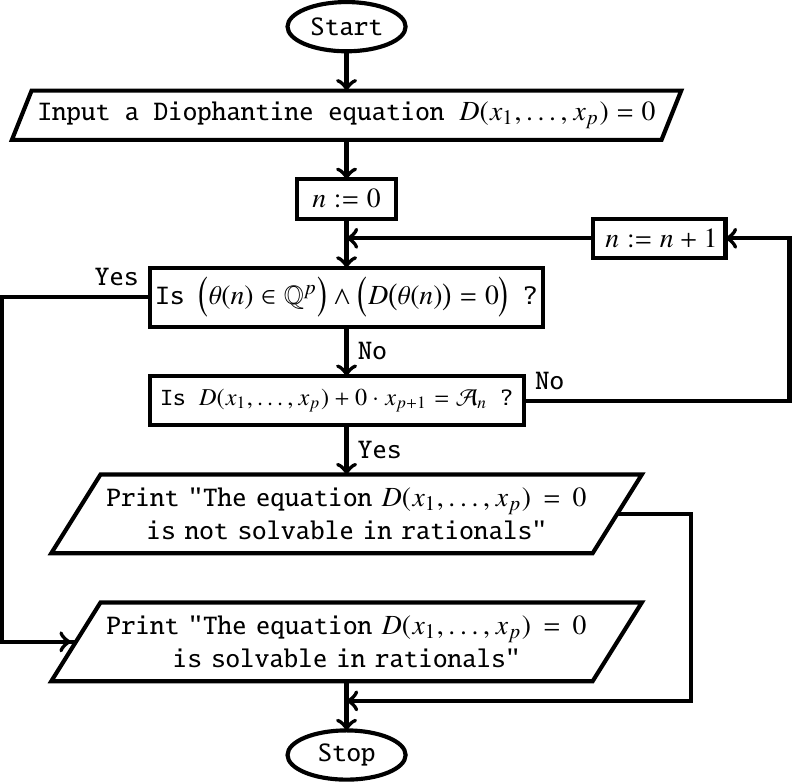}
\end{center}
\vskip 0.01truecm
\centerline{{\it Flowchart 5: An algorithm which decides the solvability of}}
\vskip 0.01truecm
\centerline{{\it a Diophantine equation \mbox{$D(x_1,\ldots,x_p)=0$} in rationals,}}
\vskip 0.01truecm
\centerline{{\it if the set of all Diophantine equations which have at most}}
\vskip 0.01truecm
\centerline{{\it finitely many rational solutions is recursively enumerable}}
\end{proof}
\begin{theorem}
A positive solution to Hilbert's Tenth Problem for $\Q$ implies that Friedman's conjecture
is false.
\end{theorem}
\begin{proof}
Assume a positive solution to Hilbert's Tenth Problem for $\Q$.
The algorithm presented in Flowchart 6 stops if and only if a Diophantine
equation \mbox{$D(x_1,\ldots,x_p)=0$} has at most finitely many rational
solutions.
\begin{center}
\includegraphics[width=\hsize]{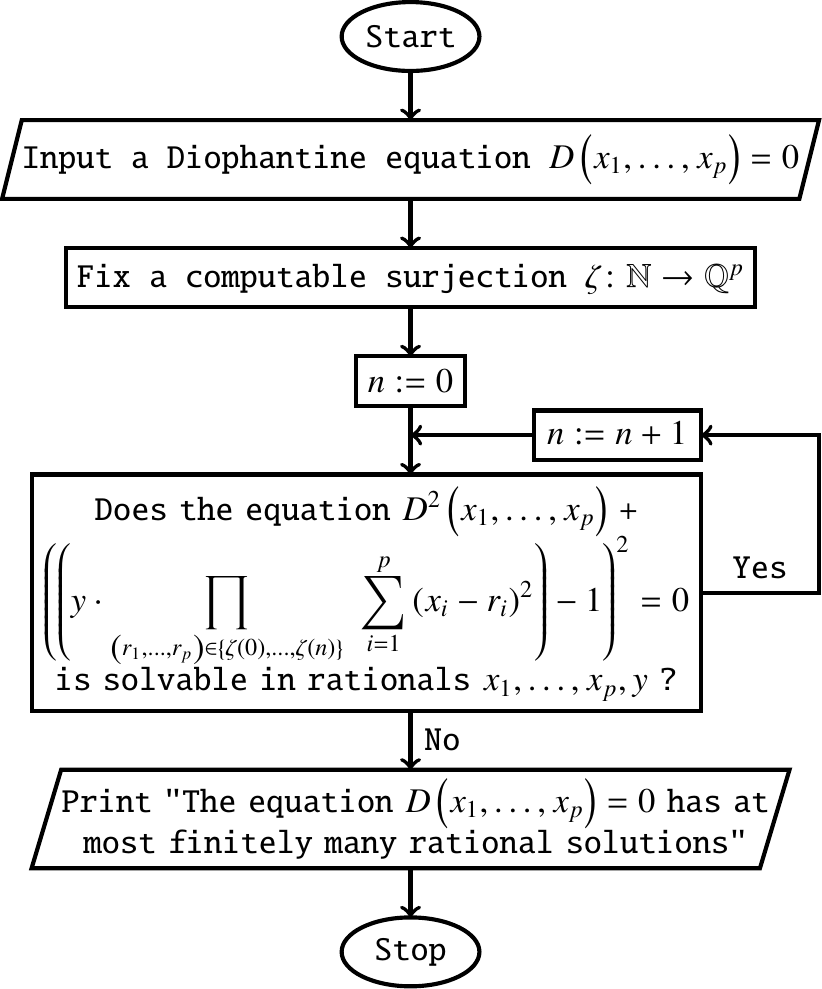}
\end{center}
\vskip 0.01truecm
\centerline{{\it Flowchart 6: A positive solution to Hilbert's Tenth}}
\vskip 0.01truecm
\centerline{{\it Problem for $\Q$ implies that Friedman's conjecture is false}}
\end{proof}
\section{{\bf A new proof that the set of all Diophantine equations
which have at most finitely many solutions in \mbox{non-negative} integers is not recursively enumerable}}
There is no algorithm to decide whether or not a given
Diophantine equation has an integer solution, see \cite{Matiyasevich}.
The set of all Diophantine equations which have at most finitely many solutions in
\mbox{non-negative} integers is not recursively enumerable, see
\mbox{\cite[p.~104,~Corollary~1]{Smorynski1}} and \mbox{\cite[p.~240]{Smorynski2}}.
\begin{lemma}\label{xyz1}
A Diophantine equation \mbox{$D(x_1,\ldots,x_p)=0$} has no solutions in \mbox{non-negative} integers \mbox{$x_1,\ldots,x_p$}
if and only if the equation \mbox{$D(x_1,\ldots,x_p)+0 \cdot x_{p+1}=0$} has at most finitely many
solutions in \mbox{non-negative} integers \mbox{$x_1,\ldots,x_{p+1}$}.
\end{lemma}
\begin{lemma}\label{xyz2}
A Diophantine equation \mbox{$D(x_1,\ldots,x_p)=0$} has no solutions in \mbox{non-negative} integers \mbox{$x_1,\ldots,x_p$}
if and only if the equation \mbox{$\left(2x_{p+1}+1\right) \cdot D(x_1,\ldots,x_p)=0$} has at most finitely many solutions
in \mbox{non-negative} integers \mbox{$x_1,\ldots,x_{p+1}$}. If the polynomial \mbox{$D(x_1,\ldots,x_p)$} depends
on all the variables \mbox{$x_1,\ldots,x_p$}, then the polynomial \mbox{$\left(2x_{p+1}+1\right) \cdot D(x_1,\ldots,x_p)$}
depends on all the varaiables \mbox{$x_1,\ldots,x_{p+1}$}.
\end{lemma}
\begin{theorem}
If the set of all Diophantine equations which have at most finitely many solutions in \mbox{non-negative} integers is recursively enumerable,
then there exists an algorithm which decides whether or not a given Diophantine equation has a solution in \mbox{non-negative} integers.
\end{theorem}
\begin{proof}
Suppose that \mbox{$\{{\cal S}_i=0\}_{i=2}^\infty$} is a computable sequence of all Diophantine equations which have
at most finitely many solutions in \mbox{non-negative} integers. The algorithm presented in \mbox{Flowchart 7} uses a computable
surjection from \mbox{$\N \setminus \{0,1\}$} onto \mbox{${\N}^p$}. By this and \mbox{Lemma \ref{xyz1}},
the execution of \mbox{Flowchart 7} decides whether or not a Diophantine equation \mbox{$D(x_1,\ldots,x_p)=0$} has a solution
in \mbox{non-negative} integers.
\vskip 0.001truecm
\noindent
\begin{center}
\begin{tikzpicture}[very thick]
\tt
\node[inner sep=0pt] at (3.9,14.7) {Start};
\node at (5,13.5) {Input a Diophantine equation $D(x_{1}, \ldots, x_{p}) = 0$};
\node at (4.75,12.35) {$W(x_{1}, \ldots, x_{p+1}) := D(x_{1}, \ldots, x_{p}) + 0 \cdot x_{p+1}$};
\node at (3.9,11.2) {$i := 2$};
\node at (7.55,10.6) {$i := i + 1$};
\node at (3.45,9.85) {Is $W(x_{1}, \ldots, x_{p+1}) = {\cal S}_{i}$ ?};
\node[text width=7.2cm, text centered] at (4.75,8.05) {Compute prime numbers $B_{\textstyle 1}, \ldots, B_{\textstyle n}$ and positive integers $b_{\textstyle 1}, \ldots, b_{\textstyle n}$ such that $i = B_{\textstyle 1}^{\textstyle b_1} \ldots B_{\textstyle n}^{\textstyle b_n}$ and $B_{\textstyle 1} < \ldots < B_{\textstyle n}$};
\node at (3.9,6.4) {Is $p \leqslant n$ ?};
\node at (3.9,5.15) {Is $D(b_{1}-1, \ldots, b_{p}-1) = 0$ ?};
\node[text width=8cm, text centered] at (4.95,3.6) {Print "The equation $D(x_{1}, \ldots, x_{p}) = 0$ is solvable in non-negative integers"};
\node[text width=8.4cm, text centered] at (5,1.9) {Print "The equation $D(x_{1}, \ldots, x_{p}) = 0$ is not solvable in non-negative integers"};
\node[inner sep=0pt] at (3.9,.3) {Stop};
\draw (3.9,14.7) ellipse(8mm and 3mm);
\draw (0,13.2) -- (9.7,13.2) -- (10,13.9) -- (.3,13.9) -- cycle;
\draw (1,12) rectangle (8.5,12.7);
\draw (3.3,10.9) rectangle (4.5,11.5);
\draw (6.6,10.3) rectangle (8.5,10.9);
\draw (1,9.5) rectangle (5.9,10.2);
\draw (1,7.2) rectangle (8.5,8.9);
\draw (2.7,6.1) rectangle (5.1,6.7);
\draw (1,4.8) rectangle (6.8,5.5);
\draw (.5,3) -- (8.7,3) -- (9.2,4.2) -- (1,4.2) -- cycle;
\draw (.5,1.3) -- (9,1.3) -- (9.5,2.5) -- (1,2.5) -- cycle;
\draw (3.9,.3) ellipse(8mm and 3mm);
\draw[->] (3.9,14.4) -- (3.9,13.9);
\draw[->] (3.9,13.2) -- (3.9,12.7);
\draw[->] (3.9,12) -- (3.9,11.5);
\draw[->] (3.9,10.9) -- (3.9,10.2);
\draw[->] (3.9,9.5) -- (3.9,8.9) node[midway,right] {No};
\draw[->] (3.9,7.2) -- (3.9,6.7);
\draw[->] (3.9,6.1) -- (3.9,5.5) node[midway,right] {Yes};
\draw[->] (3.9,4.8) -- (3.9,4.2) node[midway,right] {Yes};
\draw[->] (3.9,1.3) -- (3.9,.6);
\draw[->] (1,9.85) -- (0,9.85) -- (0,1.9) -- (.75,1.9);
\node[above left] at (1,9.85) {Yes};
\draw[->] (5.1,6.4) -- (10,6.4);
\node[above right] at (5.1,6.4) {No};
\draw[->] (6.8,5.15) -- (10,5.15);
\node[above right] at (6.8,5.15) {No};
\draw[->] (10,5.13) -- (10,10.6) -- (8.5,10.6);
\draw[->] (6.6,10.6) -- (3.9,10.6);
\draw[->] (8.95,3.6) -- (10,3.6) -- (10,1) -- (3.9,1);
\end{tikzpicture}
\vskip 0.01truecm
\centerline{{\it Flowchart 7: An algorithm which decides the solvability of}}
\vskip 0.01truecm
\centerline{{\it a Diophantine equation \mbox{$D(x_1,\ldots,x_p)=0$} in \mbox{non-negative}}}
\vskip 0.01truecm
\centerline{{\it integers, if the set of all Diophantine equations which have}}
\vskip 0.01truecm
\centerline{{\it at most finitely many solutions in \mbox{non-negative} integers}}
\vskip 0.01truecm
\centerline{{\it is recursively enumerable}}
\end{center}
\end{proof}
\begin{corollary}
By Matiyasevich's theorem, the set of all Diophantine equations which have at most finitely many solutions in \mbox{non-negative}
integers is not recursively enumerable.
\end{corollary}
\par
Let ${\cal M}$ denote the set of all Diophantine equations \mbox{$D(x_1,\ldots,x_p)=0$} such that
\mbox{$p \in \N \setminus \{0\}$} and the polynomial \mbox{$D(x_1,\ldots,x_p)$} depends on all the variables \mbox{$x_1,\ldots,x_p$}.
A similar reasoning with Lemma~\ref{xyz2} shows that
the set of all equations from ${\cal M}$ which have at most finitely
many solutions in \mbox{non-negative} integers is not recursively enumerable.
\section{Summary of the main theorems and conjectures}
Flowchart 8 provides an overview of the main theorems and conjectures.
\newpage
~
\newpage
\begin{center}
\includegraphics[scale=0.98]{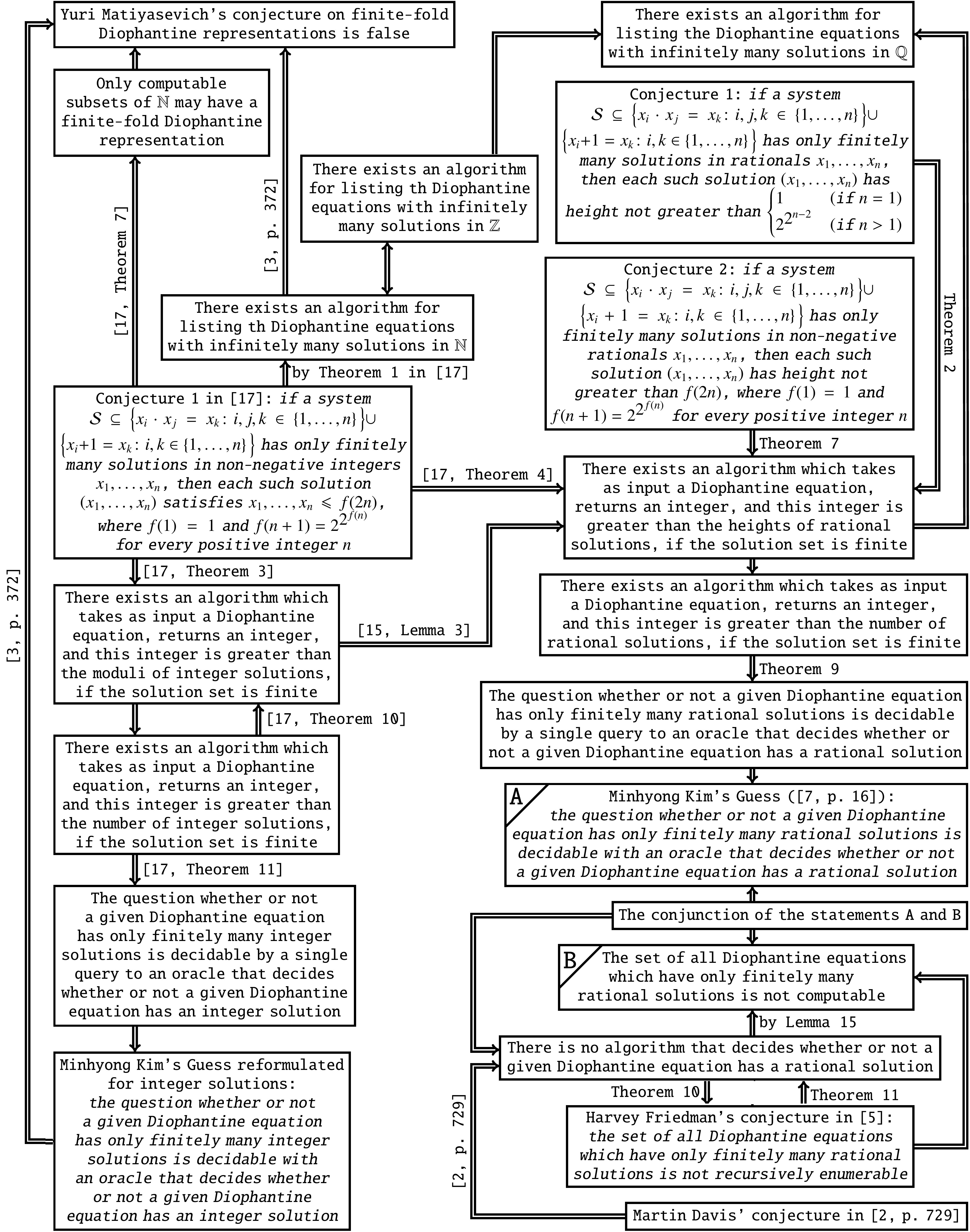}
\end{center}
\vskip 0.01truecm
\centerline{{\it Flowchart 8: Implications between conjectures}}
\newpage
~

\balance
\end{document}